\documentclass[12pt,a4paper]{article}
\usepackage{amsfonts}
\usepackage{latexsym}
\usepackage{theorem}

\newtheorem{theorem}{Theorem}[section]

\newtheorem{lemma}[theorem]{Lemma}
\newtheorem{corollary}[theorem]{Corollary}
\newtheorem{definition}[theorem]{Definition}

\newtheorem{remark}[theorem]{Remark}

\newcommand{\ben}{\begin{enumerate}}
\newcommand{\een}{\end{enumerate}}

\begin{document}
\begin{center}\large{\textbf{ ON NORM INEQUALITIES AND ORTHOGONALITY OF COMMUTATORS OF DERIVATIONS
}} \vskip 3mm \textbf{N. B. Okelo$^{1}$ and P. O. Mogotu$^{2}$} \\  $^{1}$African Institute for Mathematical Sciences,\\ 6 Melrose road, Muizenberg, 7945 Cape Town, South Africa.\\ \textbf{Email}: benard@aims.ac.za\\$^{2}$ Department of Pure and Applied Mathematics, \\ Jaramogi Oginga Odinga  University of Science and Technology,\\ Box 210-40601, Bondo-Kenya.\\\textbf{Email}: omokepriscah@yahoo.com

\end{center}

\begin{abstract}
Let $H$ be a complex separable Hilbert space and $B(H)$ the algebra of all bounded linear operators on $H$. In this paper, we give considerable generalizations of the inequalities for norms of commutators of normal operators. Let $S, T \in B(H)$ be positive normal operators with the cartesian decomposition $S=A+iC$ and $T=B+iD$ such that $a_{1}\leq A \leq a_{2},\, b_{1}\leq B \leq b_{2},\,c_{1}\leq C \leq c_{2}$ and $ d_{1}\leq D \leq d_{2}$ for some real numbers $a_{1},\,a_{2},\,b_{1},\,b_{2},\,c_{1},\,c_{2},\,d_{1}$ and $d_{2}$  we have shown that
$\|ST-TS\|\leq \frac{1}{2}\sqrt{(a_{2}-a_{1})^{2}+(c_{2}-c_{1})^{2}}\sqrt{(b_{2}-b_{1})^{2}+(d_{2}-d_{1})^{2}}.$
 Moreover, orthogonality and  norm inequalities for commutators of derivation  are also established. We have shown that if the pair of operators $(S,T)$ satisfies Fuglede-Putnam's property and $C \in ker(\delta _{S,T})$ where $C\in B(H)$ then $\|\delta _{S,T}X+C\|\geq \|C\|.$
\end{abstract}
\textbf{Keywords}: Commutator; Norm inequality; Orthogonality and Derivation

\noindent \textbf{AMS Subject Classification}: 47A12, 47A30, 47A63

\section{Introduction}
\noindent Studies on commutators and their norm inequalities have been considered by several mathematicians \cite{Kit2}, \cite{Oke} and  \cite{Mec}. Very interesting results have been obtained in special cases however, a generalization in infinite dimensional complex separable Hilbert space remain interesting.  At this point we start by defining some key terms that are useful in the sequel.
\begin{definition}
An element  $S\in B(H)$ is a commutator if there exist $X, T \in B(H)$ such that $S=XT-TX; $ is positive if $\|S\|\geq 0; $
normal if $SS^{*}=S^{*}S;$ and self-adjoint if $S=S^{*}.$ 
\end{definition}
\begin{definition}
For $S,T\in B(H),$ let $\delta _{S,T}$ denote the operator on $B(H)$ defined by $\delta _{S,T}X=SX-XT$ is called the generalized derivation. If $S=T,\,\delta _{S}X=SX-XS$ is called the inner derivation induced by $S\in B(H).$
\end{definition}

\begin{definition}\label{defn1}
Let $S,T\in B(H).$ We say that the pair $S,T$ satisfies $(FP)_{B(H)}$ the Fuglede-Putnam's property, if $SC=CT$ where $C\in B(H)$ implies $S^{*}C=CT^{*}.$
\end{definition}

\section{ Inequalities for norms of commutators}
\noindent Our aim in this section is to establish some inequalities of norms of commutators of normal operators that can be obtained naturally from cartesian decomposition and various vector inequalities in inner product spaces for example, reverse of quadratic Schwarz inequality. In the following result, we obtain a norm inequality for commutators of normal operators.
\begin{theorem}
Let $S, T \in B(H)$ be positive normal operators with the cartesian decomposition $S=A+iC$ and $T=B+iD$ such that $a_{1}\leq A \leq a_{2},\, b_{1}\leq B \leq b_{2},\,c_{1}\leq C \leq c_{2}$ and $ d_{1}\leq D \leq d_{2}$ for some real numbers $a_{1},\,a_{2},\,b_{1},\,b_{2},\,c_{1},\,c_{2},\,d_{1}$ and $d_{2}$ then,
\begin{eqnarray}\label{eqnm}
\|ST-TS\|\leq \frac{1}{2}\sqrt{(a_{2}-a_{1})^{2}+(c_{2}-c_{1})^{2}}\sqrt{(b_{2}-b_{1})^{2}+(d_{2}-d_{1})^{2}}
\end{eqnarray}
\end{theorem}
\begin{proof}
Since $S, T \in B(H)$ are normal such that $S=A+iC$ and $T=B+iD$ are the cartesian decomposition of $S$ and $T$. Then $S-z$ and $T-w$ are normal for all complex numbers $z$ and $w$ such that $a=\frac{a_{1}+a_{2}}{2},\, b=\frac{b_{1}+b_{2}}{2},\, c=\frac{c_{1}+c_{2}}{2},\, d=\frac{d_{1}+d_{2}}{2},\, z=a+ic$ and $w=b+id.$  Then
\begin{eqnarray}\label{eqnr}
 \nonumber \|ST-TS\|&=& \|(S-z)(T-w)-(T-w)(S-z)\|\\ \nonumber &\leq& \|S-z\|\|T-w\|+\|T-w\|\|S-z\|\\ &\leq& 2\|S-z\|\|B-w\|.
\end{eqnarray}
Following an analogous argument of  \cite{Sta} we have
\begin{eqnarray}\label{eqnn}
\|S-z\|^{2}\leq \|A-a\|^{2}+\|C-c\|^{2}.
\end{eqnarray}
Similarly
\begin{eqnarray}\label{eqno}
\|T-w\|^{2}\leq \|B-b\|^{2}+\|D-d\|^{2}.
\end{eqnarray}
Suppose $A,B,C,D \in B(H)$ are self-adjoint with $a_{1}\leq A \leq a_{2},\, b_{1}\leq B \leq b _{2},\,c_{1}\leq C \leq c_{2}$ and $ d_{1}\leq D \leq d_{2}$ for some real numbers $a_{1},\,a_{2},\,b_{1},\,b_{2},\,c_{1},\,c_{2},\,d_{1}$ and $d_{2}$ and $a=\frac{a_{1}+a_{2}}{2}$, then $-(\frac{a_{1}+a_{2}}{2})\leq A-a \leq \frac{a_{1}+a_{2}}{2}$ and so $\|A-a\|\leq \frac{a_{1}+a_{2}}{2}$. Similarly, $\|B-b\|\leq \frac{b_{1}+b_{2}}{2},\,\|C-c\|\leq \frac{c_{1}+c_{2}}{2}$ and $\|D-d\|\leq \frac{d_{1}+d_{2}}{2}$. Which upon substituting in Inequality \ref{eqnn} and Inequality \ref{eqno} we obtain
\begin{eqnarray}\label{eqnp}
\|S-z\|\leq \sqrt{\frac{(a_{2}-a_{1})^{2}}{2}+\frac{(c_{2}-c_{1})^{2}}{2}}.
\end{eqnarray}
and
\begin{eqnarray}\label{eqnq}
\|T-w\|\leq \sqrt{\frac{(b_{2}-b_{1})^{2}}{2}+\frac{(d_{2}-d_{1})^{2}}{2}}.
\end{eqnarray}
Substituting Inequality \ref{eqnp} and Inequality \ref{eqnq} into Inequality \ref{eqnr} we obtain $$\|ST-TS\|\leq 2\sqrt{\frac{(a_{2}-a_{1})^{2}}{2}+\frac{(c_{2}-c_{1})^{2}}{2}}\sqrt{\frac{(b_{2}-b_{1})^{2}}{2}+\frac{(d_{2}-d_{1})^{2}}{2}}$$ which upon simplification yields $$\|ST-TS\|\leq \frac{1}{2}\sqrt{(a_{2}-a_{1})^{2}+(c_{2}-c_{1})^{2}}\sqrt{(b_{2}-b_{1})^{2}+(d_{2}-d_{1})^{2}}.$$
\end{proof}
\begin{corollary}\label{cora}
Let $S, T \in B(H)$ be normal operators with the cartesian decomposition $S=A+iC$ and $T=B+iD$ such that $C$ and $D$ are positive, then
\begin{eqnarray}\label{eqns}
 \|ST-TS\|\leq \frac{1}{2}\sqrt{4\|A\|^{2}+\|C\|^{2}}\sqrt{4\|B\|^{2}+\|D\|^{2}}.
 \end{eqnarray}
\end{corollary}
\begin{proof}
 Consider Inequality \ref{eqnr} and let $a_{1}=-\|A\|,\, a_{2}=\|A\|,\, c_{1}=0,\, c_{2}=\|C\|,\, b_{1}=-\|B\|,\, b_{2}=\|B\|,\, d_{1}=0$ and $d_{2}=\|D\|.$ Substituting in Inequality  \ref{eqns} we obtain $$\|ST-TS\|\leq \frac{1}{2}\sqrt{4\|A\|^{2}+\|C\|^{2}}\sqrt{4\|B\|^{2}+\|D\|^{2}}.$$
\end{proof}
\begin{remark}
In Corollary \ref{cora} if we instead of the assumption that $C$ and $D$ are positive, we can assume that $S$ and $T$ are positive, then we obtain the inequality $$\|ST-TS\|\leq \frac{1}{2}\sqrt{\|A\|^{2}+4\|C\|^{2}}\sqrt{\|B\|^{2}+4\|D\|^{2}}.$$
\end{remark}
\begin{corollary}
Let $S \in B(H)$ with the cartesian decomposition $S=A+iC$ such that $A$ and $C$ are positive. Then $\|S^{*}S-SS^{*}\|\leq \frac{1}{2} (\|A\|^{2}+\|C\|^{2})$
\end{corollary}
\begin{proof}
Let $S \in B(H)$ has the cartesian decomposition $S=A+iC$. Also let  $A$ and $C$ be self-adjoint and $S^{*}S-SS^{*}=2i(AC-CA).$ Using [\cite{Kit2}, Inequality 36] and the arithmetic geometric mean inequality, we have $\|S^{*}S-SS^{*}\|=2\|AC-CA\|\leq \|A\|\|C\|\leq \frac{1}{2}(\|A\|^{2}+\|C\|^{2}).$
\end{proof}
\begin{theorem}
Let $S, T \in B(H)$ be normal operators with the cartesian decomposition $S=A+iC$ and $T=B+iD$ such that $a_{1}\leq A \leq a_{2},\, b_{1}\leq B \leq b_{2},\,c_{1}\leq C \leq c_{2}$ and $ d_{1}\leq D \leq d_{2}$ for some real numbers $a_{1},\,a_{2},\,b_{1},\,b_{2},\,c_{1},\,c_{2},\,d_{1}$ and $d_{2}$ and $X$ and $Y$ are compact then,
$s_{j}(SX-YT)\leq \max (\|A\|,\|B\|)s_{j}(X\oplus Y)$ for $j=1,2,...$
\end{theorem}
\begin{proof}
Let $a=\frac{a_{1}+a_{2}}{2},\, b=\frac{b_{1}+b_{2}}{2},\, c=\frac{c_{1}+c_{2}}{2},\, d=\frac{d_{1}+d_{2}}{2},\, z=a+ic$ and $w=b+id.$ We have, $(SX-YT)=(S-\frac{z+w}{2})X-Y(T-\frac{z+w}{2}).$ Taking the norms we obtain
$s_{j}\|SX-YT\|\leq \|(S-\frac{z+w}{2})\|+\|T-(\frac{z+w}{2})\|s_{j}(X\oplus Y).$ Since $S$ and $T$ are normal then $S-z$ and $T-w$ is normal. It follows by analogy that
\begin{eqnarray}\label{eqns}
\nonumber s_{j}\|SX-YT\|&\leq& \|(S-z)-\frac{z+w}{2}\|+\|(T-w)\\& &-\frac{z+w}{2})\|s_{j}(X\oplus Y).\\ \nonumber&\leq& (\|S-z\|+\|T-w\|+|z-w|)s_{j}(X\oplus Y).\\ \nonumber&\leq&(\sqrt{\|A-a\|^{2}+\|C-c\|^{2}}+\sqrt{\|B-b\|^{2}+\|D-d\|^{2}}\\ \nonumber & &+\sqrt{(a-b)^{2}+(c-d)^{2}})\; \; \; \;\;by \;\;(Equation\; \ref{eqnp}) \\ \nonumber&\leq& (\|A-a\|+\|B-b\|+\|C-c\|+\|D-d\|+|a-b|+|c-d|)s_{j}(X\oplus Y)\\ \nonumber&\leq&(\frac{a_{2}-a_{1}+b_{2}-b_{1}+c_{2}-c_{1}+d_{2}-d_{1}}{2}+\\ \nonumber & &|\frac{|a_{2}-a_{1}+b_{2}-b_{1}+c_{2}|-|c_{1}+d_{2}-d_{1}|}{2}|)s_{j}(X\oplus Y)\\ \nonumber&=&\frac{(b_{2}-a_{1})+(a_{2}-b_{1})+|(b_{2}-a_{1})-(a_{2}-b_{1})|}{2}\\ \nonumber& &+\frac{(d_{2}-c_{1})+(d_{2}-c_{1})+|(d_{2}-c_{1})-(d_{2}-c_{1})|}{2}s_{j}(X\oplus Y)\\ \nonumber &= &(\max(b_{2}-a_{1},a_{2}-b_{1})+(\max(d_{2}-c_{1},c_{2}-d_{1})s_{j}(X\oplus Y).
\end{eqnarray}
Letting $a_{1}=a_{2}=b_{1}=c_{1}=c_{2}=d_{1}=0,$ $b_{2}=\|A\|$ and $d_{2}=\|B\|$ then we have\\ $s_{j}(SX-YT)\leq \max (\|A\|,\|B\|)s_{j}(X\oplus Y)$ for $j=1,2,...$
\end{proof}
\begin{corollary}
Let $S \in B(H)$ be normal with the cartesian decomposition $S=A+iC$ if $a_{1}\leq A\leq a_{2}$ and $c_{1}\leq C\leq c_{2}$ for some real numbers $a_{1},\,a_{2},\,c_{1} $ and $c_{2}$ and if $X,Y$ are compact, then $s_{j}\|SX-YT\|\leq\|A\|s_{j}(X\oplus Y)$ for $j=1,2,...$
\end{corollary}
\begin{proof}
From Inequality \ref{eqns} replacing $T,\, b,\,d$ and $w$ by $S,\,a,\,c$ and $z$ we obtain
\begin{eqnarray*}
  s_{j}(SX-YS)&\leq& \|S-z\|+\|S-z\|\leq 2\|S-z\|\ \\ & \leq& 2\sqrt{\frac{(a_{2}-a_{1})^{2}}{2}+\frac{(c_{2}-c_{1})^{2}}{2}}s_{j}(X\oplus Y) \; \; \;\;by \;\;(Equation\; \ref{eqnp})\\ & \leq&\sqrt{(a_{2}-a_{1})^{2}+(c_{2}-c_{1})^{2}}s_{j}(X\oplus Y)
 \end{eqnarray*}
 Let $a_{1}=c_{1}=c_{2}=0$ and $\|a_{2}\|=\|A\|$ we obtain $s_{j}(SX-YS)\leq \|A\| s_{j}(X\oplus Y)$ for $j=1,2,...$
\end{proof}
\begin{lemma}\label{lem3}
Let $S, T \in B(H)$ be normal operators belonging to the norm ideal associated with the Hilbert Schmidt norm $\|.\|_{2}$ such that there product $ST$ is normal. Then
\begin{eqnarray}\label{eqnj}
\|ST\|_{2}\leq \|TS\|_{2}.
\end{eqnarray}
\end{lemma}
\begin{proof}
Let $w(A)$ denote the numerical radius of $S$. Then from  $w(S)\leq \|S\|$ and if $S$ is normal, $w(S)= \|S\|.$ Moreover, for any two normal operators $S$ and $T$   we have
\begin{eqnarray}\label{eqnk}
w(ST)= \|TS\|.
\end{eqnarray}
 Suppose $ST$  is normal, then we have;
\begin{eqnarray*}
\|ST\|_{2}\leq w(ST)=w(TS)\leq \|TS\|_{2}.
\end{eqnarray*}
\end{proof}
\begin{lemma}\label{lem4}
Let  $S$ and $T$ be as Lemma \ref{lem3} above. If $ST$ is a normal operator then $\||ST|^{\frac{1}{2}}\|\leq \||TS|^{\frac{1}{2}}\|$.
\end{lemma}
\begin{proof}
Invoking Equation \ref{eqnk} we have
$ \||ST|^{\frac{1}{2}}\|=w(|ST|^{\frac{1}{2}})= w(|TS|^{\frac{1}{2}})\leq \||TS|^{\frac{1}{2}}\|$.
\end{proof}
\begin{theorem}
Let $X$ be a positive definite operator and let $S$ and $T$ be normal operators belonging to the norm ideal associated with the Hilbert Schmidt norm $\|.\|_{2}.$ Then $\|S-T\|^{2}_{2}\leq \|SX-XT\|^{2}_{2}\|X^{-1}S-TX^{-1}\|^{2}_{2}.$
\end{theorem}
\begin{proof}
Suppose $S$ and $T$ are self-adjoint then we can write $\|S-T\|_{2}=\||(S-T)^{2}|^{\frac{1}{2}}\|_{2}=\||(S-T)X^{-\frac{1}{2}}X^{\frac{1}{2}}(S-T)|^{\frac{1}{2}}\|_{2}.$ Using Lemma \ref{lem3} and Lemma \ref{lem4} we get
\begin{eqnarray*}
\|S-T\|_{2}& \leq &\||X^{\frac{1}{2}}(S-T)^{2}X^{-\frac{1}{2}}|^{\frac{1}{2}}\|_{2}\\ &\leq &(\|X^{\frac{1}{2}}(S-T)X^{\frac{1}{2}}\|_{2}\|X^{-\frac{1}{2}}(S-T)X^{-\frac{1}{2}}\|_{2})^{\frac{1}{2}}\\ &\leq &(\|Re[(S-T)X]\|_{2}\|Re[X^{-1}(S-T)]\|_{2})^{\frac{1}{2}}.
\end{eqnarray*}
Since $TX-XT$ and $TX^{-1}-X^{-1}T$ are skew-Hermitian \cite{Wan} then\\ $Re[(S-T)X]=Re[(S-T)X+(TX-XT)]=Re(SX-XT),$ and \\$Re[X^{-1}(S-T)]=Re[X^{-1}(S-T)+(TX^{-1}-X^{-1}T)]=Re(X^{-1}S-TX^{-1}).$ This implies that $\|S-T\|_{2}\leq (\|Re(SX-XT)\|_{2}\|Re(X^{-1}S-TX^{-1})\|_{2})^{\frac{1}{2}}.$ But $\|Re S\|\leq \|S\|$ for any operator. So we have $\|S-T\|_{2}\leq (\|SX-XT\|_{2}\|X^{-1}S-TX^{-1}\|_{2})^{\frac{1}{2}}$ which upon squaring we obtain the required result.
 \end{proof}
 \begin{corollary} Let $S,T,X \in B(H)$ such that $S$ and $T$ are positive, then $\|SX-XT\|_{2}\|\leq \|X\|_{2}(\|S\|^{2}_{2}+\|T\|^{2}_{2})^{\frac{1}{2}}.$
 \end{corollary}
 \begin{proof}
 $\|SX-XT\|_{2}=\|(S-T)X\|_{2}\leq \|S-T\|_{2}\|X\|_{2}$. Since $S$ and $T$ are positive i.e $\|S-T\|_{2}=(\|S\|^{2}_{2}+\|T\|^{2}_{2})^{\frac{1}{2}}$, then we have by \cite{Kit2} \\$\|SX-XT\|_{2}\leq \|X\|_{2}(\|S\|^{2}_{2}+\|T\|^{2}_{2})^{\frac{1}{2}}.$
 \end{proof}
 \begin{theorem}
 Let $S, T \in B(H)$ be positive and self-adjoint operators and $ST-TS$ be also positive. If $n> 0$ is defined such that $\|ST-TS\|\leq n$, then
 \begin{eqnarray}\label{eqni}
  \|STx\|^{2}\geq \frac{1}{n^{2}}(\|STx\|^{4}-|\langle (ST)^{2}x,x\rangle|^{2}) \,\, \forall x \in H,\,, \|x\|=1.
  \end{eqnarray}
 \end{theorem}
 \begin{proof}
  We employ the reverse of quadratic Schwarz inequality in  i.e. $$0\leq \|a\|^{2}\|b\|^{2}-|\langle a,b\rangle|^{2}\leq \frac{1}{|\alpha|^{2}}\|a\|^{2}\|a-\alpha b\|^{2}.$$ For every $a, b \in H$, let $\alpha=1\, \,, a=ST\,, b=TS$, we have $$\|STx\|^{2}\|TSx\|^{2}-|\langle STx,TSx\rangle|\leq \|STx\|^{2}\|STx-TSx\|^{2}.$$ Since $ST=TS$, we have $\|STx\|^{4}-|\langle (ST)^{2}x,x\rangle|\leq n^{2}\|STx\|^{2}$ which upon simplification yield, $\|STx\|^{2}\geq \frac{1}{n^{2}}(\|STx\|^{4}-|\langle (ST)^{2}x,x\rangle|^{2}).$
$\frac{1}{2}$ is the best constant possible in Inequality \ref{eqni} in the sense that it cannot be replaced by a smaller quantity in general. The equality case is realized in Inequality \ref{eqni} if, for instance one takes $S=\left(
                       \begin{array}{cc}
                         1 & 1 \\
                         1 & -1 \\
                       \end{array}
                     \right) $ ,\,  $T=\left(
                       \begin{array}{cc}
                         0 & 1 \\
                         1 & 0 \\
                       \end{array}
                     \right) $ and a unit vector  $x=\left(
                       \begin{array}{cc}
                         0  \\
                         1  \\
                       \end{array}
                     \right) .$ Therefore Inequality \ref{eqni} becomes\\ $\|STx\|^{2}\geq \frac{1}{2}(\|STx\|^{4}-|\langle (ST)^{2}x,x\rangle|^{2}) \,\, \forall x \in H,\,, \|x\|=1.$
\end{proof}

 \section{Orthogonality of commutators of derivations}
 \noindent In this section, we give some new results on orthogonality of commutators of normal derivation with respect to Fuglede-Putnam's property and norm-attainable operators.
 \begin{lemma}\label{lemk}
 Let $S,T,C\in B(H)$. Then the following are equivalent
 \begin{itemize}
  \item [i.] The pair $(S,T)$ has the property $(FP)_{B(H)}.$
  \item [ii.] If $SC=CT$, then $R(C)$ reduces $S$, ker$(C)^{\bot}$ reduces $T$ and $S|_{\overline{R(C)}}$ and $T|_{ker(C)^{\bot}}$ are normal operators where $R$ and the $ker$ denote the range and the kernel.
  \end{itemize}
 \end{lemma}
 \begin{proof}
 $(1)\Rightarrow (2)$ Analogously by the proof of \cite{Mec}  Since $SC=CT$ and the pair $(S,T)$ has the property $(FP)_{B(H)}$, $S^{*}C=CT^{*}$ this implies that $\overline{R(C)}$ and $ker(C)^{\bot}$ are the reducing subspaces for $S$ and $T$. If $S(SC)=(SC)T$, by $(FP)_{B(H)}$ we obtain $S^{*}(SC)=(SC)T^{*}$ and the identity $S^{*}C=CT^{*}$ implies that $S^{*}SC=SS^{*}C$. This shows that $S|_{\overline{R(C)}}$ is normal. Indeed, $(T^{*},S^{*})$ satisfies $(FP)_{B(H)}$ and $(T^{*}C^{*}=C^{*}S^{*}$. Similarly $T^{*}|_{\overline{R(C)}}=(T|_{ker{(C)}})^{*}$\\ $(2)\Rightarrow (1)$ If $ C\in B(H)$ such that $SC=CT.$ Let $S=S_{1}\oplus S_{2}$ with respect to the orthogonal decomposition $H=\overline{R(C)}\oplus \overline{R(C)^{\bot}},\,\,T=T_{1}\oplus T_{2}$ with respect to $H=ker(C)\oplus ker(C)^{\bot}$ and $X: \overline{R(C)}\oplus \overline{R(C)^{\bot}}\rightarrow ker(C)^{\bot}\oplus ker(C)$ have the matrix representation $X=\left[
                                                                                                 \begin{array}{cc}
                                                                                                   X_{1} & X_{2} \\
                                                                                                   X_{3} & X_{4}\\
                                                                                                 \end{array}
                                                                                               \right]$ From $SC=CT$, it follows that $S_{1}C_{1}=C_{1}T_{1}$. Since $S_{1}$ and $T_{1}$ are normal operators, then applying the Fuglede-Putnam's property, we obtain $S^{*}_{1}C_{1}=C_{1}T^{*}_{1}$ which implies that $S^{*}C=CT^{*}.$

 \end{proof}
\begin{theorem}\label{thm4}
 Let $S,T,X\in B(H)$. If the pair of operators $(S,T)$ satisfies Fuglede-Putnam's property and $C \in ker(\delta _{S,T})$ where $C\in B(H)$ then $\|\delta _{S,T}X+C\|\geq \|C\|.$
\end{theorem}
\begin{proof}
Since the pair $(S,T)$ satisfies the $(FP)_{B(H)}$ property it follows that from Lemma \ref{lemk} that $\overline{R(C)}$ reduced $S,\,\,ker^{\bot}(C)$ reduces $T$ and $S|_{\overline{R(C)}},\,\,T|_{ker^{\bot}(C)}$ are normal operators.  Letting $C_{o}: ker^{\bot}(C)\rightarrow \overline{R(C)}$ be the quasi-affinity defined by setting $C_{1}x=Cx$ for each $x \in ker^{\bot}$, then it results that $\delta_{S,T}(C_{o})=\delta_{S^{*}_{1},T^{*}_{1}}(C_{o})=0.$ By Lemma \ref{lemk}, we have the matrix representation $S=\left[
                                                                                                                         \begin{array}{cc}
                                                                                                                           S_{1} & 0 \\
                                                                                                                           0 & S_{2} \\
                                                                                                                         \end{array}
                                                                                                                       \right],\,\,T=\left[
                                                                                                                         \begin{array}{cc}
                                                                                                                           T_{1} & 0 \\
                                                                                                                           0 & T_{2} \\
                                                                                                                         \end{array}
                                                                                                                       \right],\,\,C=\left[
                                                                                                                         \begin{array}{cc}
                                                                                                                           C_{1} & 0 \\
                                                                                                                           0 & C_{2} \\
                                                                                                                         \end{array}
                                                                                                                       \right],\,\,X=\left[
                                                                                                                         \begin{array}{cc}
                                                                                                                           X_{1} & X_{2} \\
                                                                                                                           X_{3} & X_{4} \\
                                                                                                                         \end{array}
                                                                                                                       \right]$ Since $S_{1}$ and $T_{1}$ are two normal operators, then $\|\delta_{S,T}(X)+C\|=\|\left(
                                                                                                                                                  \begin{array}{cc}
                                                                                                                                                    \delta _{S_{1},T_{1}}X+C_{1} & * \\
                                                                                                                                                    * &* \\
                                                                                                                                                  \end{array}
                                                                                                                                                \right)
                                                                                                                       \|\geq \|\delta_{S_{1},T_{1}}(X)+C_{1}\|=\|C\|.$

\end{proof}
 \begin{corollary}
 Let $S,T,X\in B(H)$ and $C \in ker(\delta _{S,T})$ then $\|\delta _{S,T}X+C\|\geq \|C\|.$
 \end{corollary}
\begin{proof}
On $H\oplus H$ consider the operator $M, N$ and $Y$ defined as: $$N=\left[
                                                                                                                         \begin{array}{cc}
                                                                                                                           S & 0 \\
                                                                                                                           0 & T \\
                                                                                                                         \end{array}
                                                                                                                       \right],\,\,M=\left[
                                                                                                                         \begin{array}{cc}
                                                                                                                           0 & C \\
                                                                                                                           0 & 0 \\
                                                                                                                         \end{array}
                                                                                                                       \right],\,\,Y=\left[
                                                                                                                         \begin{array}{cc}
                                                                                                                           0 & X \\
                                                                                                                           0 & 0 \\
                                                                                                                         \end{array}
                                                                                                                       \right].$$

                                                                                                                       Then $N$ is normal, $M \in {N}^{'}$ and $\delta _{N}(Y)+M=\left[
                                                                                                                         \begin{array}{cc}
                                                                                                                           0 & \delta _{S,T}X+C \\
                                                                                                                           0 & 0 \\
                                                                                                                         \end{array}
                                                                                                                       \right].$

 Applying Theorem \ref{thm4} to the operators $N,M$ and $Y$ and $M \in \delta _{S,T}$ and $\|\delta _{N}Y+M\|\geq \|M\|.$ Therefore, $C \in \delta _{S,T}$ and $\|\delta _{S,T}X+C\|\geq \|C\|.$
\end{proof}
\section{Conclusion}
In this paper, we have given results on norm inequality for commutators and also orthogonality of these commutators in Banach algebras.

\section{Acknowledgement}
The first author is grateful to AIMS South Africa for the visiting research fellowship to carry out this research. The second author is grateful for the financial support from NRF Kenya Grant No.: NRF/2016/2017/0034.

\end{document}